\documentclass[11pt]{article}

\usepackage{amsmath,amssymb,amsthm,setspace,graphicx}
\usepackage[text={6.5in,8.4in}, top=1.3in]{geometry}
\usepackage{hyperref}

\newtheorem{maintheorem}{Theorem}
\newtheorem{theorem}{Theorem}[section]
\newtheorem{proposition}[theorem]{Proposition}
\newtheorem{lemma}[theorem]{Lemma}

\newtheorem{definition}[theorem]{Definition}

\newcommand{\T}{{\cal T}}
\newcommand{\R}{{\cal R}}
\newcommand{\B}{{\cal B}}
\renewcommand{\P}{{\cal P}}

\begin{document}

\title{All feedback arc sets of a random Tur\'an tournament have
	$\lfloor n/k\rfloor-k+1$ disjoint $k$-cliques (and this is tight)
	\thanks{This research was supported by the Israel Science Foundation (grant No. 1082/16).}}

\author{
Safwat Nassar
\thanks{Department of Mathematics, University of Haifa, Haifa
31905, Israel.}
\and
Raphael Yuster
\thanks{Department of Mathematics, University of Haifa, Haifa
31905, Israel. Email: raphy@math.haifa.ac.il}
}

\date{}

\maketitle

\setcounter{page}{1}

\begin{abstract}

What must one do in order to make acyclic a given oriented graph?
Here we look at the structures that must be removed (or reversed) in order to make acyclic a given oriented graph.

For a directed acyclic graph $H$ and an oriented graph $G$,
let $f_H(G)$ be the maximum number of pairwise disjoint copies of $H$ that can be found in {\em all}
feedback arc sets of $G$. In particular, to make $G$ acyclic, one must at least remove (or reverse) $f_H(G)$ pairwise disjoint copies of $H$. Perhaps most intriguing is the case where $H$ is a $k$-clique, in which case the parameter is denoted by $f_k(G)$. Determining $f_k(G)$ for arbitrary $G$ seems challenging.

Here we essentially answer the problem, precisely, for the family of $k$-partite tournaments.
Let $s(G)$ denote the size of the smallest vertex class of a $k$-partite tournament $G$.
It is not difficult to show that $f_k(G) \le s(G)-k+1$ (assume that $s(G) \ge k-1$).

Our main result is that for all sufficiently large $s=s(G)$, there are $k$-partite tournaments for which
$f_k(G) = s(G)-k+1$. In fact, much more can be said: a random $k$-partite tournament $G$
satisfies $f_k(G) = s(G)-k+1$ almost surely (i.e. with probability tending to $1$ as $s(G)$ goes to infinity).
In particular, as the title states, $f_k(G) = \lfloor n/k\rfloor-k+1$  almost surely, where
$G$ is a random orientation of the Tur\'an graph $T(n,k)$.

\vspace*{3mm}

\end{abstract}

\section{Introduction}

All graphs in this paper are finite and simple. An {\em orientation} (a.k.a. oriented graph) is obtained by assigning a direction to each
edge of an undirected graph.
Important classes of oriented graphs are {\em tournaments} which are
orientations of a complete graph and {\em multipartite tournaments} which are
orientations of complete multipartite graphs. A multipartite tournament with $k$ parts is a {\em $k$-partite tournament}. For further results and applications of tournaments
and multipartite tournaments see the textbook \cite{BJ-2008}.
An oriented graph without directed cycles is called a {\em directed acyclic graph}.
In this paper {\em acyclic graphs} always refers to directed acyclic graphs.

A natural meta question, studied by various researchers, is the ``complexity'' of an orientation in terms of
its directed cycles. In other words, what must one do in order to make acyclic a given orientation?
In order to address such questions one must inevitably look at feedback arc sets.

For a directed graph $G$, a {\em feedback arc set} is a set of edges covering every directed cycle.
Equivalently, it is a spanning subgraph whose complement is acyclic.
A feedback arc set is {\em minimal} if removing an edge from it results in a non-feedback arc set.
Let $F(G)$ be the set of all minimal feedback arc sets of $G$ and let $A(G)$ denote the set of maximal acyclic subgraphs of $G$. Observe that $A(G)$ consists of the complements of the elements of $F(G)$. Another simple property to observe is that all elements of $F(G)$ are themselves acyclic.
Consequently, an element of $F(G)$ has the property that reversing (instead of removing) its edges
also converts $G$ to an acyclic graph.

The combinatorial and computational aspects of $F(G)$ and its complement $A(G)$ have been studied quite extensively, both for general digraphs as well as for tournaments. Let us just mention here that for a tournament $T$,
the set $F(T)$, in general, has a complicated structure; indeed it is NP-Hard to find an element of $F(T)$ with minimum size \cite{alon-2006}. Similarity, $A(T)$ has a complicated structure;
while it is trivial that there are elements of $A(T)$ whose chromatic number is at least $\sqrt{|T|}$,
it is an open problem to determine (in terms of $|T|$) the asymptotics of the largest chromatic number of an element of $A(T)$ \cite{AHSRT-2013,FKS-preprint,NY-2019}.
There are quite a few nontrivial questions on $F(T)$ and $A(T)$ when considering
random tournaments. For example, almost surely,
the minimum size of a feedback arc set of a random $n$-vertex tournament is
$0.5\binom{n}{2}-Cn^{3/2}+o(n^{3/2}) $ but the exact value of $C$ is not known \cite{delavega-1983,spencer-1971}.

This paper is about a basic problem on the structure of feedback arc sets.
For an acyclic graph $H$ and an oriented graph $G$,
let $f_H(G)$ be the maximum number of pairwise disjoint copies of $H$ that can be found in {\em all}
feedback arc sets of $G$ (if $H$ is not acyclic then clearly $f_H(G)=0$ as the minimal feedback arc sets, being acyclic, do not contain $H$).
In particular, to make $G$ acyclic, one must always at least remove (or, equivalently, reverse the edges of) $f_H(G)$ pairwise disjoint copies of $H$. Perhaps most intriguing is the case where $H$ is a very dense object,
namely, a $k$-clique, in which case the parameter is denoted by $f_k(G)$.
Notice that an acyclic $k$-clique is the (unique) transitive tournament on $k$ vertices.
Determining $f_k(G)$ for arbitrary $G$ seems challenging.

If $G$ is $k$-chromatic, then trivially $f_{k+1}(G)=0$, but how large can $f_k(G)$ be, and are there cases
where it is always very large? To answer this question we must look at the densest $k$-chromatic orientations,
namely at $k$-partite tournaments. For a $k$-partite tournament $G$, let $s(G)$ denote the size of its smallest vertex class. As each $k$-clique of $G$ is a transversal of the vertex classes, obviously there are at most
$s(G)$ vertex-disjoint $k$-cliques in $G$, so $f_k(G) \le s(G)$ is straightforward. But recall that we
are not simply asking for vertex-disjoint $k$-cliques in $G$. Rather, we seek a much stronger requirement:
we are asking for vertex-disjoint $k$-cliques in {\em all feedback arc sets} of $G$.
In fact, it is not difficult to prove that $f_k(G) \le s(G)-k+1$ (or $f_k(G)=0$ if $s(G) \le k-1$).
Indeed, suppose that the vertex classes of $G$ are $V_1,\ldots,V_k$, where $|V_k|=s(G)$.
Let $u_1,\ldots,u_{k-1}$ be $k-1$ distinct vertices of $V_k$.
Consider the spanning subgraph of $G$ obtained by removing all the edges connecting
$u_i$ to all the vertices of $V_i$.
It is clearly a feedback arc set as its complement is a forest of stars.
Hence, there is a minimal feedback arc set in which $u_i$ does not appear in any $k$-clique.
So, the maximum number of pairwise disjoint $k$-cliques we can find in this minimal feedback arc set is
$s(G)-k+1$.

But does it get any worse than that? Or perhaps the above construction is a ``worst example'' in the sense that the obstacle is that there are these $k-1$ vertices in the smallest part which are forced to isolation with another part
in some feedback arc set? Indeed, a corollary of our main result is that for most $k$-partite tournaments,
it cannot get any worse than that. To state our main result we need to recall the notion of random orientations.
Let ${\cal T}(n_1,\ldots,n_k)$ denote the probability space of all $k$-partite tournaments with $n_i$ vertices in the $i$th vertex class. That is, the orientation of each edge is selected uniformly at random, and each choice is independent of all other choices. Assume by symmetry that $n_k = \min_{i=1}^k n_i$. Hence,
for each $G \sim {\cal T}(n_1,\ldots,n_k)$ we have $s(G)=n_k$.
\begin{maintheorem}\label{t:1}
	Let $G \sim {\cal T}(n_1,\ldots,n_k)$. With probability $1-o_{n_k}(1)$ it holds that
	$f_k(G) = s(G)-k+1$.
\end{maintheorem}
Hence, not only does there exist a $k$-partite tournament $G$ with $f_k(G) = s(G)-k+1$, but, in fact, most $k$-partite tournaments are such. Observe also that it is trivially inevitable to speak about {\em most} $k$-partite tournaments  and not {\em all} $k$-partite tournaments. Indeed, to take it to extreme, if $G$ itself is an acyclic $k$-partite tournament, then its unique feedback arc set is the empty graph and $f_k(G)=0$ is this case.
Hence, Theorem \ref{t:1} is, in this sense, best possible not only in the exact value $f_k(G)=s(G)-k+1$, but also
in the statement about it holding {\em almost} always.

It is worth pointing out the special case where the $n_i$'s are as equal as possible.
Indeed, as the expression $s(G)-k+1$ only involves the size of the smallest vertex class, it suffices to prove
the lower bound of Theorem \ref{t:1} in the most difficult case where all the $n_i$'s are as equal as possible.
Let therefore ${\cal T}(n,k)$ denote the probability space of all orientations of the Tur\'an graph $T(n,k)$
(such orientations are called {\em Tur\'an tournaments}).
So here the total number of vertices is $n$, each vertex class is of size $\lfloor n/k \rfloor$ or
$\lceil n/k \rceil$, and $G \sim {\cal T}(n,k)$ has $s(G)=\lfloor n/k \rfloor$.
So equivalently, it suffices to prove that for $G \sim {\cal T}(n,k)$, with probability $1-o_{n}(1)$ it holds that
$f_k(G) = \lfloor n/k \rfloor-k+1$.

Perhaps what makes proving Theorem \ref{t:1} rather involved is that although we are
looking at a random structure (in this case the symmetric probability space of all $k$-partite tournaments),
when we want to say that some statement holds for {\em all} feedback arc sets, we must account for the fact that
surely there are many minimal feedback arc sets that look very far from typical random objects (for example,
for sure there are always these minimal feedback arc sets with isolated vertices).
Also, there are many examples in extremal graph theory where although some obvious obstacles for containing a spanning subgraph (say, isolated vertices) are simple to identify, it is not easy to prove that these are the {\em only} obstacles \cite{AY-2013,GPW-2012,ore-1961}. We should also mention the connection of our result (and proof) to the notion of
{\em local resilience} of random graphs \cite{SV-2008} (see also \cite{BLS-2012} concerning local resilience of
random graphs with respect to almost triangle factors). In our proofs, we repeatedly use the fact that {\em all}
minimal feedback arc sets of a typical $G \sim {\cal T}(n,k)$ have many vertices of sufficiently high degree.
Now, if it were true that there are no vertices of small degree, or even if it were true that those vertices of small
degree are ``non-obstacles'', then we could have applied methods from the area of local resilience of random graphs to prove that each such minimal feedback arc set has the required amount of disjoint $k$-cliques. Unfortunately, the sufficiently non-random structure of some minimal feedback arc sets of $G$ makes these (small amount of) small degree vertices
into true obstacles and we do not see a way of using random graph local resilience results as an alternative way to prove Theorem \ref{t:1}.

The case $k=2$ of Theorem \ref{t:1} has a fairly routine proof, which we present in Section 3 as a warm-up. However, the proof for $k \ge 3$ is by far more involved and its proof
comprises Sections 4,5,6. In Section 4 we set up some particular properties that are guaranteed to exist with high probability in a random Tur\'an tournament. Sections 5 and 6 assume, therefore, that a
Tur\'an tournament is given with all these properties. In Section 5 we build (using probabilistic arguments) an absorber that
will help us to gradually build the $k$-cliques from smaller $r$-cliques for $r=2,\ldots,k$ without ``getting stuck'' in the process. The iterative process itself is described in Section 6.
The next section sets up some preliminaries used throughout the paper.

\section{Some preliminaries}

As mentioned in the introduction, it suffices to prove the lower bound in Theorem \ref{t:1} for the probability space of Tur\'an
tournaments, namely when each vertex class is of size $\lfloor n/k \rfloor$ or $\lceil n/k \rceil$.
Furthermore, it suffices to prove it in the case where $k | n$. Indeed, if not,
then we can just remove a single vertex from each part with $\lceil n/k \rceil$ vertices so that all
parts have size $\lfloor n/k \rfloor$. Equivalently, instead of repeatedly using $n/k$ each time,
it is slightly more convenient to assume that each part has size $n$ and there are $k$ parts.
We therefore let $\R(n,k)=\T(nk,k)$ denote the probability space of all $k$-partite tournaments with $n$ vertices in each part.
An equivalent formulation of Theorem \ref{t:1} which will be more convenient to prove is therefore the following:
\begin{theorem}\label{t:1-equiv}
	Asymptotically almost surely, $T \sim \R(n,k)$ has the property that every element of $F(T)$ has
	$n-k+1$ pairwise disjoint $k$-cliques.
\end{theorem}

The following are some standard notations that will be repeatedly used.
Edges of directed graphs are denoted by ordered pairs $(u,v)$ where $v$ is called an {\em out-neighbor of $u$} and $u$ is called an {\em in-neighbor} of $v$. Edges of undirected graphs are denoted as $uv$.
If $G$ is a (directed or undirected) graph, $V(G)$ and $E(G)$ denote its vertex set and edge set, respectively. For $T \sim \R(n,k)$, we let $N^+(v)$ denote the out-neighbors of $v$ 
and let $N^-(v)$ denote the in-neighbors of $v$. If we wish to consider out-degrees and in-degrees in
some subgraph $H$, we will use the notations $N_H^+(v)$ and $N_H^-(v)$.

We will require a few large deviation inequalities of random variables.
The proof of each of them can be found in the book
\cite{AS-2004}. The first one is the Chernoff bound on the concentration of the binomial distribution
$\B(m,p)$ where $p$ is fixed.
\begin{lemma}\label{l:chernoff}
	Let $0 < p < 1$. Suppose $X \sim \B(m,p)$, then for every $a > 0$,
	$$
	\Pr[|X -mp| > a] < 2e^{-a^2/m}\;.
	$$
\end{lemma}

The second one is a special case of Janson's inequality.
\begin{lemma}\label{l:janson} 
	Let $X_1,\ldots,X_m$ be indicator random variables, each $X_i$ corresponding to the event that a certain subset $V_i$ of vertices of a random graph induces some fixed subgraph. Suppose that the success probability of each $X_i$ is $p$.
	Let $\Delta$ be the number of ordered pairs $(X_i,X_j)$ such that $|V_i \cap V_j| \ge 2$.
	Then for every  $\gamma > 0$,
	$$
	\Pr\left[\sum_{i=1}^m X_i \le (1-\gamma)mp\right] < e^{-\gamma^2 mp/(2+\frac{\Delta}{m})}\;.
	$$
\end{lemma}

The third one is Azuma's inequality.
\begin{lemma}\label{l:azuma}
	Let $Y_0,\ldots,Y_m$ be a martingale with $|Y_{i+1}-Y_i| \le 1$ for all $0 \le i < m$. Then for
	all $\lambda > 0$,
	$$
	\Pr\left[Y_m < Y_0 -\lambda \sqrt{m}\right] < e^{-\lambda^2/2}\;.
	$$
\end{lemma}

\section{The bipartite case}

In this section we prove the case $k=2$ of Theorem \ref{t:1}, stated in its equivalent formulation
Theorem \ref{t:1-equiv} that a random element of $\R(n,2)$ has the property that all of its feedback arc sets have a matching of size $n-1$. Specifically we prove:
\begin{proposition}\label{p:k=2}
Asymptotically almost surely, $T \sim \R(n,2)$ has the property that every element of $F(T)$ has
a matching of size $n-1$.
\end{proposition}

Before presenting the proof we establish some notation.
We denote the vertex parts of elements of $\R(n,2)$ by $A_1$ and $A_2$ where $|A_1|=|A_2|=n$.
Let ${\cal S}_V$ denote the set of permutations of a set $V$.
Let $\pi \in {\cal S}_{V(T)}$ denote a permutation of the vertices $V(T) = A_1 \cup A_2$ of $T \sim \R(n,2)$.

Recall from the introduction that if $H \in F(T)$ is a minimal feedback arc set, then $H$
is acyclic. In particular, there is a topological sort of its vertices,
namely, a permutation $\pi \in {\cal S}_{V(T)}$ such that $(u,v) \in E(T)$ has the property that
$\pi(u) < \pi(v)$ if and only if $(u,v) \in E(H)$. Hence, we can associate with each
$H \in F(T)$ a permutation $\pi$ as above (observe that $\pi$ is not necessarily unique,
as a topological sort is not necessarily unique). More conveniently,
\begin{definition}[$L_\pi(T)$]
For $\pi \in {\cal S}_{V(T)}$ let $L_\pi(T)$ be the spanning ordered acyclic subgraph of $T$ where the vertices are ordered according to $\pi$ and which consists of all the edges $(u,v) \in E(T)$ with $\pi(u) < \pi(v)$ (edges that go from ``left to right''). 
\end{definition}
As the orientation of the edges of $L_\pi(T)$ is determined from $\pi$, we may view $L_\pi(T)$ as an undirected graph.
Observe also that there is an onto correspondence from the set of all $|V(T)|!$ possible $L_\pi(T)$ to $F(T)$. It will be more convenient to prove Theorem \ref{t:1-equiv} as well as Proposition \ref{p:k=2}
by considering all possible  $L_\pi(T)$. So, in particular, proving proposition \ref{p:k=2} amounts
to proving that for $T \sim \R(n,2)$ it holds a.a.s. that for {\em all} $\pi \in {\cal S}_{V(T)}$
the graph $L_\pi(T)$ has a matching of size $n-1$.

One minor difficulty in proving this is that, although it is very easy to prove that
for any given $\pi \in {\cal S}_{V(T)}$, an element $T \sim \R(n,2)$ has with very high probability
the property that $L_\pi(T)$ has a matching of order $n-1$ (in fact, in most cases, a perfect matching), this probability is not high enough so
as to apply the union bound of the complement event for all $(2n)!$ permutations.
In fact, it is not difficult to show that a.a.s. $T \sim \R(n,2)$ has $|F(T)|$ of order $2^{\Theta(n \log n)}$ so we cannot ``save much'' by just considering a representative $\pi$ for each minimal feedback arc set.

To overcome this obstacle, we need to first establish a few properties that are guaranteed to exist in $T \sim \R(n,2)$
with high probability. To state these properties we require further notation.

Let $R \subseteq A_1$ and let $S \subseteq A_2$. For a permutation $\pi$
let $X(T,\pi,R,S)$ denote the event that in $L_\pi(T)$, there is no edge between $R$ and $S$.
Observe that $\bigcup_{\pi \in {\cal S}_{V(T)}} X(T,\pi,R,S)$ is just the event that $T[R \cup S]$ is acyclic.

\begin{definition}[$D$-consistency]\label{d:consistency}
For a vector $D \in \{+,-\}^d$ and for a sequence $A'$ of $d$ distinct vertices of $T$, let
$C_D(A')$ be the subset of vertices of $T$ that are {\em $D$-consistent with $A'$}.
Namely, the $j$th vertex of $A'$ has all the vertices of $C_D(A')$ as its out-neighbors in $T$ if $D(j)$ is plus and has all the vertices of $C_D(A')$ as its in-neighbors in $T$ if $D(j)$ is minus.	
\end{definition}

\noindent
For example, if $D=(+,-)$ then $C_D((u,v))$ is $N^+(u) \cap N^-(v)$.

\begin{lemma}\label{l:1}
	For all $n$ sufficiently large the following holds for $T \sim \R(n,2)$.
	\begin{enumerate}
		\item
		$$
		\Pr\left[\bigcup_{R \subseteq A_1,|R| \ge n/20} ~ \bigcup_{S \subseteq A_2,|S| \ge n/20} ~ \bigcup_{\pi \in {\cal S}_{V(T)}}
 		X(T,\pi,R,S)\right] \le \frac{1}{n}\;.
 		$$
 		\item
 		Let $D \in \{+,-\}^2$ and let $(u,v)$ be an ordered pair of two distinct vertices.\
 		$$
 		\Pr\left[|C_D((u,v))| \ge 1.1n/4 \right] \le \frac{1}{n^3}\;.
 		$$
\end{enumerate}
\end{lemma} 
\begin{proof}
As the orientation of each edge of $T$ connecting a vertex of $R$ with a vertex of $S$ is chosen uniformly
and independently at random, we have for any given $\pi$ that $\Pr[X(T,\pi,R,S)] = 2^{-|R||S|}$.
As there are at most $2^n$ choices for $S$, at most $2^n$ choices for $R$, and $(2n)!$ choices for $\pi$
we have that for $n$ sufficiently large, the union event in the first statement of the lemma holds with probability at most
$$
2^n \cdot 2^n \cdot (2n)! 2^{-n^2/400} \le \frac{1}{n}\;.
$$
For the second statement of the lemma, fix $D \in \{+,-\}^2$ and
fix the ordered pair $(u,v)$. We may assume that $u,v$ are either both in $A_1$ or both in $A_2$ as otherwise $C_D((u,v))= \emptyset$.
Therefore, $|C_D((u,v))|$ is a random variable which is the sum of $n$ independent indicator
random variables with success probability $1/4$.
Hence, its distribution is $\B(n,1/4)$.
By Lemma \ref{l:chernoff}, $|C_D((u,v))|$ satisfies the claimed bound with probability at most $1/n^3$.
\end{proof}

Notice that the number of events corresponding to Item 2 in Lemma \ref{l:1}
is only $O(n^2)$ as there are only $4$ choices for $D$ and at most $n^2$
choices for $(u,v)$ where $u,v$ are distinct vertices both from the same part.
Now, since $1-1/n-O(n^2)/n^3 \ge 1-o_n(1)$ it follows that an element
$T \sim \R(n,2)$ satisfies both properties that correspond to the items
of Lemma \ref{l:1} with probability $1-o_n(1)$:\\
{\bf Property 1:} $X(T,\pi,R,S)$ does not occur for each $\pi$ and for each choice of
$R \subseteq A_1$ and $S \subseteq A_2$ satisfying $|R| \ge n/20$, $|S| \ge n/20$.\\
{\bf Property 2:} $|C_D((u,v))| \le (1.1)n/4$ for each choice of
$D \in \{+,-\}^2$ and for each ordered pair $(u,v)$ of two distinct vertices of $T$.

Consider any $T \sim \R(n,2)$ satisfying these two properties.
We will prove that for every $\pi \in {\cal S}_{V(T)}$, the graph $L_\pi(T)$ has a matching of 
size $n-1$. Observe that this is a completely deterministic claim.
So, from now until the end of this section, fix $\pi$ and fix $T \sim \R(n,2)$ satisfying Properties
1 and 2.

\begin{lemma}\label{l:3}
	There is at most one vertex $a \in A_1$ whose degree in $L_\pi(T)$ is smaller than $n/18$.
	Similarly, there is at most one vertex $b \in A_2$ whose degree in $L_\pi(T)$ is smaller than $n/18$.
\end{lemma}
\begin{proof}
	As both claims are analogous, we prove the first one.
	Assume otherwise, that there are two distinct vertices $a,a' \in A_1$ each with fewer than $n/18$ neighbors in $L_\pi(T)$.
	Let $B^*$ be the set of vertices of $A_2$ that are non-neighbors of both $a$ and $a'$ in $L_\pi(T)$.
	Then, by our assumption, $|B^*| \ge n - n/9=8n/9$.
	
	Suppose without loss of generality that $\pi(a') < \pi(a)$.
	Each vertex of $A_2$ is positioned in $\pi$ either before $a'$, between $a'$ and $a$ or after $a$.
	Let $B_1$ be those vertices of $A_2$ positioned before $a'$,
	$B_2$ be those positioned between $a'$ and $a$,
	and $B_3$ be those positioned after $a$.
	
	What can we say about $B_1 \cap B^*$? The reason for a vertex $b \in A_2$ positioned in $\pi$ before $a$ and before $a'$ to be a non-neighbor of both of them in $L_\pi(T)$ is that $(a',b) \in E(T)$ and $(a,b) \in E(T)$ (observe: these edges point from ``right to left'' so are not in $L_\pi(T)$).
	In other words, $B_1 \cap B^* \subseteq N^+(a') \cap N^+(a)$.
	But by Property 2 with the ordered pair $(a',a)$ and $D=(+,+)$ we have that $|B_1 \cap B^*| \le 1.1n/4$.
	Similarly, $|B_2 \cap B^*| \le 1.1n/4$ using Property 2 with $D=(-,+)$, and
	$|B_3 \cap B^*| \le 1.1n/4$ using Property 2 with $D=(-,-)$.
	But this implies that $|B^*|=|B^* \cap (B_1 \cup B_2 \cup B_3)| \le 3.3n/4 < 8n/9$, a contradiction.
\end{proof}

By Lemma \ref{l:3}, we can remove from $L_\pi(T)$ one vertex of $A_1$ and one vertex of $A_2$ such that
the bipartite induced subgraph $L'_\pi(T)$ of $L_\pi(T)$ obtained after removal has minimum degree at least
$\delta(L'_\pi(T))=t \ge n/18-1 \ge n/19$. We next prove the following:
\begin{lemma}\label{l:case-1}
	$L'_\pi(T)$ has a perfect matching.
\end{lemma}
\begin{proof}
We prove the lemma using Hall's Theorem.
Let $A'_1$ and $A'_2$ be the sides of $L'_\pi(T)$ and observe that $|A'_1|=|A'_2|=n-1$.
We must therefore show that for all $R \subseteq A'_1$,
$|N_{A'_2}(R)| \ge |R|$ where $N_{A'_2}(R)$ are the set of vertices of $A'_2$ for which there is an edge of $L_\pi(T)$ connecting them to
a vertex of $R$.

Suppose first that $|R| \le t$. In this case we have $|N_{A_2'}(R)| \ge t \ge |R|$ as the minimum degree of $L'_\pi(T)$ is $t$.

If $t < |R| \le n-1-t$ then set $S=A_2' \setminus N_{A_2'}(R)$, so there is no edge in $L_\pi(T)$ between $R$ and $S$. Observe that $|R| \ge n/19$ since $|R| > t$.
Since $X(T,\pi,R,S)$ holds, we must have by Property 1 that $|S| < n/20$.
But this implies that $|N_{A_2'}(R)| > (n-1)-n/20 \ge n-1-t \ge |R|$.

Finally, if $|R| > n-1-t$ then $N_{A_2'}(R)=A_2'$ since the minimum degree is $t$.
\end{proof}
Observe that since $L'_\pi(T)$ has a perfect matching, $L_\pi(T)$ has
a matching of size $n-1$. As this holds for all $\pi \in {\cal S}_{V(T)}$ for a $T$ satisfying properties 1-2, and as $T \sim \R(n,2)$ satisfies Properties 1 and 2 with probability $1-o_n(1)$, 
Proposition \ref{p:k=2} follows. \qed

We end this section be pointing out the major obstacle when trying the extend the bipartite case to the $k$-partite case.
One natural way to try to do this is by induction. Suppose we have already found $n-t+1$ pairwise-disjoint transitive $t$-cliques on the vertices of the first $t$ parts $A_1,\ldots,A_t$. We now expose the (randomly directed) edges incident with part $A_{t+1}$ having their other endpoint in $\cup_{j=1}^t A_j$. One can define a bipartite graph with one side being
the (already found) cliques and the other part being $A_{t+1}$, and an edge connects $v \in A_{t+1}$ with a clique if the addition of $v$ to that clique extends to a transitive $(t+1)$-clique. The goal would then be to show that this bipartite graph has a matching of size $n-t$. To this end, we require an analogue of Lemma \ref{l:1} and an analogue of Lemma
\ref{l:3}. Now, while an analogue of Lemma \ref{l:1} is relatively easy to obtain, there is no similar analogue for Lemma \ref{l:3}.
In fact, one cannot just fix the $n-t+1$ transitive $t$-cliques and expect such an extension to $t+1$ cliques.
Thus, one needs to ``plan ahead'' and have sufficient flexibility to perform an inductive step. This is what motivates the use of an appropriately defined absorber.

\section{Establishing properties}

In the remaining sections we prove the cases $k \ge 3$ of Theorem \ref{t:1-equiv} (which, recall, is equivalent to Theorem \ref{t:1}). In this section we establish several properties, some are quite delicate,  that are possessed with high probability by $T \sim \R(n,k)$. In Sections 5 and 6, we prove that
$T \sim \R(n,k)$ which possesses these properties, has $n-k+1$ pairwise disjoint $k$-cliques in each of its feedback arc sets.

We require some generalization of notations introduced in the previous section.
As each element of $\R(n,k)$ is $k$-partite, we denote the parts by $A_1,\ldots,A_k$ where
$|A_i|=n$.
Let $T \sim \R(n,k)$ and let $\pi$ be a permutation of the vertices $V(T) = \cup_{i=1}^k A_i$ of $T$,
hence $\pi \in {\cal S}_{V(T)}$. Recall from the previous section that proving Theorem \ref{t:1-equiv} amounts to proving that for $T \sim \R(n,k)$, it holds a.a.s.
that for all $\pi \in {\cal S}_{V(T)}$, the graph $L_\pi(T)$ has a $n-k+1$ pairwise disjoint $k$-cliques.

\begin{definition}[Perfect $r$-set]
	Let $1 \le r \le k$. A {\em perfect $r$-set} is a set $P$ of $n$ pairwise disjoint $r$-tuples where the
	$i$'th element of each $r$-tuple is from $A_i$. The set of all $(n!)^{r-1}$ perfect $r$-sets is denoted by $\P_r$.
\end{definition}
Note: while each element of a perfect $r$-set $P$ induces an $r$-clique in $T$,
we do not require in the definition that such an element induces an $r$-clique in any given $L_\pi(T)$.

\begin{definition}[$L_\pi(P,T)$]
Let $1 \le r < k$. Given a perfect $r$-set $P$ and given $\pi \in  {\cal S}_{V(T)}$, define the (undirected) bipartite graph $L_\pi(P,T)$ as follows. One part of $L_\pi(P,T)$ is $P$ and the other part is $A_{r+1}$.
Observe that each part has $n$ vertices.
The edges of $L_\pi(P,T)$ are defined as follows.
Consider some pair $(p,v)$ where $p=(a_1,\ldots,a_r) \in P$ and $v \in A_{r+1}$.
Then $pv$ is an edge of $L_\pi(P,T)$ if and only if for all $i=1,\ldots,r$, $\{v,a_i\}$
induces an edge of $L_\pi(T)$ (the orientation of each of these $r$ edges in $L_\pi(T)$ is not important).
\end{definition}

\begin{definition}[The event $X(T,\pi,P,R,S)$]
Let $1 \le r < k$. For $\pi \in  {\cal S}_{V(T)}$, for a perfect $r$-set $P$, for $R \subseteq P$,
for $S \subseteq A_{r+1}$ and for $T \sim \R(n,k)$ let $X(T,\pi,P,R,S)$ denote the event that in $L_\pi(P,T)$ there are fewer than
$|R||S|/2^{r+1}$ edges between $R$ and $S$.
\end{definition}

\begin{lemma}\label{l:X-property}
	Let $\epsilon > 0$ be given and let $1 \le r < k$.
	For all sufficiently large $n$ as a function of $\epsilon,k$ the following holds for $T \sim \R(n,k)$:
		$$
		\Pr\left[\bigcup_{P \in {\cal P}_r} ~ \bigcup_{R \subseteq P,|R| \ge \epsilon n} ~\bigcup_{S \subseteq A_{r+1},|S| \ge \epsilon n} ~ \bigcup_{\pi \in {\cal S}_{V(T)}}
		X(T,\pi,P,R,S)\right] \le \frac{1}{kn}\;.
		$$
\end{lemma}
\begin{proof}
	Fix $\pi,P,R,S$ and let $T \sim \R(n,k)$. Consider some pair $(p,v)$ with
	$p=(a_1,\ldots,a_r) \in P$ and $v \in A_{r+1}$.
	The probability that $pv$ is an edge in $L_\pi(P,T)$ is $1/2^r$ as it depends on the orientation of
	the $r$ edges of $T$ connecting $v$ with $a_1,\ldots,a_r$.
	As for any two distinct pairs $(p,v)$ and $(p',v')$
	we have either $v \neq v'$ or else $p \cap p' = \emptyset$, the event that $pv$ is an edge of $L_\pi(P,T)$ is independent of any other combination of events of the form $p'v'$. So, the number of edges between $R$
	and $S$ in $L_\pi(P,T)$ has distribution ${\cal B}(|R||S|,2^{-r})$.
	By Lemma \ref{l:chernoff}, the probability that
	this random variable falls by a constant factor below its expectation, in particular falls
	below $|R||S|/2^{r+1}$, is exponentially small
	in $|R||S|$. Hence, by the assumption on the sizes of $R$ and $S$ being at least $\epsilon n$ in the stated expression, the probability of the event $X(T,\pi,P,R,S)$ is $2^{-\Theta(n^2)}$.
	
	There are $(kn)!$ choices for $\pi$, $(n!)^{r-1}$ choices for $P$,
	and at most $2^n$ choices for each of $R$ and $S$. Altogether, the number of choices of the $4$-tuple  $(\pi,P,R,S)$
	is only $2^{\Theta(n \log n)}$. Hence, for $n$ sufficiently large as a function of $\epsilon$ and $k$, the union event in the statement of the lemma holds with probability at most $1/(kn)$.
\end{proof}

\begin{definition}\label{d:friendly-clique}[Friendly $r$-clique]
	Let $1 \le r < k$.
	Suppose that $\{v_1,\ldots,v_r\}$ induce an $r$-clique in $L_\pi(T)$ where $v_i \in A_i$.
	We say that this $r$-clique is {\em friendly} if for every $r < t \le k$,
	and for every $1 \le r' \le r$ the number of
	vertices of $A_t$ that are common neighbors of $v_1,\ldots,v_{r'}$ in $L_\pi(T)$ is at least $n/2^{r'+1}$. Otherwise, the $r$-clique is {\em unfriendly}. 
\end{definition}

\begin{definition}[The event $Y(T,\pi,S_1,\ldots,S_r)$]
	Let $1 \le r < k$. For $\pi \in  {\cal S}_{V(T)}$, for subsets $S_i \subseteq A_i$ for $i=1,\ldots,r$
	of size at least $n/18$ each, and for $T \sim \R(n,k)$
	let $Y(T,\pi,S_1,\ldots,S_r)$ denote the event that in $L_\pi(T)$ there are fewer than
	$0.5(1/18)^rn^r2^{-\binom{r}{2}}$ friendly $r$-cliques induced by $\cup_{i=1}^r S_i$.
\end{definition}
\begin{lemma}\label{l:Y-property}
	Let $1 \le r \le k-1$.
	For all sufficiently large $n$ as a function of $k$ the following holds for $T \sim \R(n,k)$:
	$$
	\Pr\left[\bigcup_{\pi \in {\cal S}_{V(T)}} ~ \bigcup_{i=1}^{r} ~ \bigcup_{S_i \subseteq |A_i|, |S_i| \ge n/18}
	Y(T,\pi,S_1,\ldots,S_r)\right] \le \frac{1}{kn}\;.
	$$
\end{lemma}
\begin{proof}
	Let $\mu=1/18$. Fix $\pi,S_1,\ldots,S_r$ where $|S_i| \ge \mu n$ and let $T \sim \R(n,k)$.
	As $Y(T,\pi, S_1,\ldots,S_r)$ implies $Y(T,\pi, S_1^{*},\ldots,S_r^{*})$
	if $S_i^{*} \subseteq S_i$, we may assume $|S_i| = \mu n$.
	Consider some $r$-tuple of vertices $(a_1,\ldots,a_r)$ with $a_i \in S_i$.
	The probability that $\{a_1,\ldots,a_r\}$ induces an $r$-clique in $L_\pi(T)$ is $2^{-\binom{r}{2}}$.
	Hence if $Z$ denotes the set of (not necessarily friendly) $r$-cliques in $L_\pi(T)$ induced by
	$\cup_{i=1}^r S_i$, then $|Z|$ is the sum of $\mu^rn^r$ indicator variables, each with
	success probability $2^{-\binom{r}{2}}$. Thus, $E[|Z|]=\mu^rn^r2^{-\binom{r}{2}}$.
	Each indicator variable corresponding to $(a_1,\ldots,a_r)$ is independent of
	a variable corresponding to $(b_1,\ldots,b_r)$ if they intersect in at most one vertex (have no pair in common).
	Hence each such variable is independent of all other variables but at most
	$r^2\mu^{r-2}n^{r-2}$. Hence, by Lemma \ref{l:janson}
	with $p=2^{-\binom{r}{2}}$, $m=\mu^rn^r$, $\Delta \le r^2\mu^{r-2}n^{2r-2}$, $\gamma = 1/4$,
	the probability that $|Z|$ is smaller than $0.75 \mu^rn^r2^{-\binom{r}{2}}$ is at most $e^{-\Theta(n^2)}$.
	
	Given that $|Z| \ge 0.75\mu^rn^r2^{-\binom{r}{2}}$, what is the probability that fewer than
	$0.5\mu^rn^r2^{-\binom{r}{2}}$ of the elements of $Z$ are friendly? If this has occurred, then there are
	are at least $0.25\mu^rn^r2^{-\binom{r}{2}}$  unfriendly $r$-cliques. In particular,
	as each vertex can only be in $\Theta(n^{r-1})$ $r$-cliques, there are $\Theta(n)$ pairwise
	disjoint unfriendly $r$-cliques. What is the probability of an $r$-clique induced by $\{a_1,\ldots,a_r\}$
	to be unfriendly? Let $r < t \le k$ and let $1 \le r' \le r$. The number of common neighbors of $a_1,\ldots,a_{r'}$ in
	$A_t$ is distributed $\B(n,2^{-r'})$ hence, by Lemma \ref{l:chernoff}, the probability that this number falls below
	$n/2^{r'+1}$ is $e^{-\Theta(n)}$. So, the probability of being unfriendly is
	at most $k^2e^{-\Theta(n)} = e^{-\Theta(n)}$. Hence, the probability of $\Theta(n)$ pairwise disjoint $r$-cliques to be all unfriendly is $e^{-\Theta(n^2)}$. As the number of sets of pairwise disjoint
	$r$-cliques is smaller than $n^n$, the probability that $Z$ has more than $0.25\mu^rn^r2^{-\binom{r}{2}}$  unfriendly $r$-cliques remains $e^{-\Theta(n^2)}$.
	So, the probability of the event $Y(T,\pi,S_1,\ldots,S_r)$ occurring is $e^{-\Theta(n^2)}$.
	
	There are $(kn)!$ choices for $\pi$ and at most $2^n$ choices for each $S_i$.
	Altogether, the number of choices of the tuple $(\pi,S_1,\ldots,S_r)$ 
	is only $2^{\Theta(n \log n)}$. Hence, for $n$ sufficiently large as a function of $k$, the union event in the statement of the lemma holds with probability at most $1/(kn)$.
\end{proof}	

We recall, and then extend, the notion of consistency from the previous section.
Let $1 \le r < k$. For an $r$-tuple $p=(a_1,\ldots,a_r)$ with $a_i \in A_i$, 
a vertex $v \in A_t$ with $r < t \le k$ and a vector $W \in \{+,-\}^{r}$ we say that $v$ is {\em $W$-consistent with $p$} if the following holds for each $i=1,\ldots,r$:
$(v,a_i) \in E(T)$ if and only if $W(i)$ is plus (otherwise $(a,v_i) \in E(T)$).
If $v$ is not $W$-consistent with $p$, it is {\em $W$-inconsistent with $p$}.
Clearly, for $T \sim \R(n,k)$, given $p$, $W$, and $v$, the probability that $v$ is $W$-inconsistent with $p$ is $1-1/2^r$. We need to extend the notion of $W$-inconsistency to higher dimensions as follows.

\begin{definition}[$\hat{W}$-inconsistent with $\hat{p}$]
Let $1 \le r < k$.
For a sequence $\hat{p}=(p_1,\ldots,p_d)$ of $d$ pairwise-disjoint $r$-tuples as above
and for a sequence $\hat{W} = (W_1,\ldots,W_d)$ of $d$ vectors each from $\{+,-\}^{r}$, 
a vertex $v \in A_t$ where $r < t \le k$ is {\em $\hat{W}$-inconsistent with $\hat{p}$}
if for all $1 \le i \le d$, $v$ is $W_i$-inconsistent with $p_i$.
Let $I_{\hat{W}}(\hat{p},t)$ be the subset of vertices of $A_t$ that are $\hat{W}$-inconsistent
with $\hat{p}$.
\end{definition}
Clearly, since the $p_i$ in $\hat{p}$ are pairwise disjoint, given $\hat{p}$, $\hat{W}$, and $v \in A_t$,
we have that for $T \sim \R(n,k)$ the probability that $v$ is $\hat{W}$-inconsistent with $\hat{p}$ (or, equivalently, $v \in I_{\hat{W}}(\hat{p},t)$) is
$(1-1/2^r)^d$.

\begin{lemma}\label{l:inconsistent-property}
	Let $d$ be a positive integer, let $1 \le r < k$ and let $r < t \le k$.
	Let $\hat{W} = (W_1,\ldots,W_d)$ be as above and let $\hat{p}=(p_1,\ldots,p_d)$ be as above.
	For all sufficiently large $n$ as a function of $k,d$ the following holds for $T \sim \R(n,k)$:
	$$
	\Pr\left[|I_{\hat{W}}(\hat{p},t)| \ge n(1-1/2^r)^d+n^{2/3} \right] \le \frac{1}{n^{kd+1}}\;.
	$$
\end{lemma}
\begin{proof}
Notice that $|I_{\hat{W}}(\hat{p},t)|$ is a random variable which is the sum of $n$ independent indicator
random variables with success probability $(1-1/2^r)^d$. The independence follows from the fact that for
each $v \in A_t$, the probability that $v \in I_{\hat{W}}(\hat{p},t)$ is independent of all other events corresponding to other vertices of $A_t$ as it depends only the orientations of edges of $T$ incident with $v$. Hence $|I_{\hat{W}}(\hat{p},t)|$ is distributed $\B(n,(1-1/2^r)^d)$.
Thus, by Lemma \ref{l:chernoff}, the probability that it deviates from its expected value, which is linear in $n$, by more than an additive term of $n^{2/3}$ is exponentially small in a polynomial in $n$
(recall: $d,k$ are fixed and $r,t \le k$). In particular, $|I_{\hat{W}}(\hat{p})|$ satisfies the
claimed bound with probability at most $n^{-kd-1}$.
\end{proof}

We may merge Lemmas \ref{l:X-property}, \ref{l:Y-property}, and \ref{l:inconsistent-property} together with the notion of $D$-consistency from Definition \ref{d:consistency} into the following lemma.
\begin{lemma}\label{l:all-properties}
		Let $d$ be a positive integer and let $\epsilon$ be a positive real.
		Then with probability $1-o_n(1)$ the following properties hold for $T \sim \R(n,k)$:
		\begin{itemize}
		\item[]{\bf Property 1:} For all $1 \le r < k$,
		$X(T,\pi,P,R,S)$ does not occur for all $\pi \in {\cal S}_{V(T)}$, for all $P \in {\cal P}_r$, for all $R \subseteq P$ with $|R| \ge \epsilon n$ and for all
		$S \subseteq A_{r+1}$ with $|S| \ge \epsilon n$.
		\item[]{\bf Property 2:} For all $1 \le r < k$,
		$Y(T,\pi,S_1,\ldots,S_r)$ does not occur for all $\pi \in {\cal S}_{V(T)}$,
		and for all $r$-tuples $(S_1,\ldots,S_r)$ with $S_i \subseteq A_i$ and $|S_i| \ge n/18$.
		\item[]{\bf Property 3:} For all $1 \le r < k$, for all $r < t \le k$,
		for all $\hat{W} = (W_1,\ldots,W_d)$ where $W_i \in \{+,-\}^{r}$ and for all $\hat{p}=(p_1,\ldots,p_d)$ where the $p_j$ are pairwise disjoint $r$-tuples
		$p_j=(a_{j,1},\ldots,a_{j,r})$ with $a_{j,i} \in A_i$, it holds that
		$|I_{\hat{W}}(\hat{p},t)| \le n(1-1/2^r)^d+n^{2/3}$.
		\item[]{\bf Property 4:} For all $1 \le s \le k$, for all $1 \le \ell \le k$, 
		for all sequences $A'=(v_1,\ldots,v_{q})$ of $q \le d$ distinct elements of $A_\ell$, and for all $D \in \{+,-\}^{q}$, it holds that $|C_{D}(A') \cap A_s | \le n(1/2)^{q}+n^{2/3}$.
		\end{itemize}
\end{lemma}
\begin{proof}
	By Lemma \ref{l:X-property}, Property 1 does not hold with probability at most $1/(kn)$ for each $r$,
	hence it does not hold with probability at most $1/n$ for all $r$.
	By Lemma \ref{l:Y-property}, Property 2 does not hold with probability at most $1/(kn)$ for each $r$,
	hence it does not hold with probability at most $1/n$ for all $r$.
	As for Property 3, the number of possible $\hat{W}$ is $2^{rd}$. The number of possible
	$\hat{p}$ is smaller than $n^{rd}$. As the probability that
	$|I_{\hat{W}}(\hat{p},t)| \ge n(1-1/2^r)^d+n^{2/3}$ is at most $n^{-kd-1}$ by Lemma \ref{l:inconsistent-property}, Property 3 does not hold with probability $O(1/n)$.
	For Property 4, observe that for a given
	sequence $(v_1,\ldots,v_{q})$ of distinct elements of $A_\ell$ and for a given $D \in \{+,-\}^{q}$, the random variable $|C_{D}(A') \cap A_s|$ is distributed $\B(n,(1/2)^q)$ (the case $s=\ell$ is trivial).
	By Lemma \ref{l:chernoff}, it does not satisfy the claimed bound with probability exponentially small in $n^{1/3}$, in particular with probability smaller than $n^{-d-1}$.
	As there are $k$ choices for $s$, $k$ choices for $\ell$, fewer than $n^d$ choices for a sequence $(v_1,\ldots,v_{q})$ from $A_\ell$, and $2^{q}$ choices for $D$,
	Property 4 does not hold with probability $O(1/n)$.
	Hence, all four properties simultaneously hold with probability $1-O(1/n)$.	
\end{proof}

By Lemma \ref{l:all-properties}, in order to complete the proof of Theorem \ref{t:1-equiv},
it remains to prove the following (completely deterministic) lemma.
\begin{lemma}\label{l:main}
There exists a positive integer $d$ and a real $\epsilon > 0$
such that the following holds for all $n$ sufficiently large.
For every tournament $T \sim \R(n,k)$ satisfying properties 1-4 of Lemma \ref{l:all-properties}
and for every $\pi \in {\cal S}_{V(T)}$,
$L_\pi(T)$ has a set of at least $n-k+1$ pairwise disjoint $k$-cliques.
\end{lemma}

We next define a series of constants that will be used in the remainder of the proof, in particular, we define the required $d,\epsilon$ for which Lemma \ref{l:main} holds.
\begin{definition}[The constants]\label{d:constants}
	$ $\\
	(i) ~~ $\mu=1/18$.\\
	(ii) ~	$\rho = 0.25\mu^k 2^{-\binom{k}{2}}$.\\
	(iii) ~Let $d$ be the smallest integer satisfying for all $r=1,\ldots,k$
	$$
	\frac{\mu-d^{1-2r}}{dr+1} > (1-2^{-r})^d ~~~~{\rm and}~~~~\frac{d^{2-2r}-d^{1-2r}}{d+1} > \frac{1}{2^d}\;.
	$$
	(iv) ~~$\delta= \min\{\rho/k^2~,~1/(2kd^{2k})\}$.\\
	(v) ~~\;$\epsilon = \delta \rho/5$.
\end{definition}

From here until the end of Section 6, we fix $\pi$ and fix $T \sim \R(n,k)$ satisfying
Properties 1-4 of Lemma \ref{l:all-properties}.
Hence, we omit $\pi$ and $T$ in the definitions and notations that follow
(previous definitions and notations that use $\pi$ and $T$ remain the same).
In the various claims that follow we will always assume that $n$ is sufficiently large
as a function of the constants in Definition \ref{d:constants}, hence as a function of
$d,\epsilon,k$ (and, therefore, in fact, as a function of $k$). To prove Lemma \ref{l:main} we need to prove that $L_\pi(T)$ has a set of at least $n-k+1$ pairwise disjoint $k$-cliques.

\begin{definition}[Friendly vertex]
	A vertex $v \in A_r$ is a {\em friendly vertex} if for all $i=1,\ldots,k$, $i \neq r$,
	the number of vertices of $A_i$ that are neighbors of $v$ in $L_\pi(T)$ is at least $n/17$.
\end{definition}
The following is a generalization of Lemma \ref{l:3}.
\begin{lemma}\label{l:friendly}
	For every $1 \le r \le k$, there are at most $k-1$ vertices of $A_r$ that are not friendly vertices.
\end{lemma}
\begin{proof}
	Suppose $v_1,\ldots,v_k$ are $k$ non-friendly vertices from $A_r$.
	Then for each $v_i$, there is some $j \neq r$ such that the 
	number of vertices of $A_j$ that are neighbors of $v_i$ in $L_\pi(T)$ is smaller than $n/17$.
	Hence, there are two distinct vertices, say $v_1,v_2$, and some $j \neq r$
	such that in $L_\pi(T)$, each of them has fewer than $n/17$ neighbors from $A_j$.
	Let $B^* \subseteq A_j$ be the non-neighbors in $L_\pi(T)$ of both of them.
	Then, $|B^*| \ge 15n/17$.
	
	Suppose $\pi(v_1) < \pi(v_2)$. Each vertex of $A_j$ is positioned in $\pi$ either before $v_1$, between $v_1$ and $v_2$ or after $v_2$. Let $B_1$ be those vertices of $A_j$ positioned before $v_1$,
	$B_2$ be those positioned between $v_1$ and $v_2$, and $B_3$ be those positioned after $v_2$.
	What can we say about $B_1 \cap B^*$? The reason for a vertex $b \in A_j$ positioned in $\pi$ before $v_1$ and before $v_2$ to be a non-neighbor of both of them in $L_\pi(T)$ is that $(v_1,b) \in E(T)$ and $(v_2,b) \in E(T)$.
	But by Property 4 with the sequence $(v_1,v_2)$ and $D=(+,+)$ we have that $|B_1 \cap B^*| \le n/4 + n^{2/3}$.
	Similarly, $|B_2 \cap B^*| \le n/4+n^{2/3}$ using Property 4 with $D=(-,+)$, and
	$|B_3 \cap B^*| \le n/4+n^{2/3}$ using Property 4 with $D=(-,-)$.
	But this implies that $|B^*|=|B^* \cap (B_1 \cup B_2 \cup B_3)| \le 3n/4+3n^{2/3} < 15n/17$, a contradiction.
\end{proof}

For $1 \le r \le k$, fix $A^*_r$ to be a set of $n-k+1$ friendly vertices of $A_r$.
The process of constructing the $n-k+1$ pairwise disjoint $k$-cliques proceeds as follows.
We will, in fact, construct a perfect $k$-set $P$. This $k$-set will have the
property that all but $k-1$ of its elements induce $k$-cliques.
Furthermore, the $n-k+1$ vertices from each $A_r$ that belong the $k$-cliques of $P$ are
precisely $A^*_r$.
So, we can view the $n-k+1$ pairwise disjoint $k$-cliques that we construct as a perfect matching in the
$k$-uniform $k$-partite hypergraph ${\cal H}$, whose parts are $A_1^*,\ldots,A_k^*$ and whose ``edges'' are
the $k$-cliques they induce. Notice however that since $\pi$ is arbitrary, this hypergraph can be quite
far from resembling a random $k$-partite $k$-uniform hypergraph. Different vertices can have very different degrees
(the difference can be $\Theta(n^{k-1})$) in this hypergraph. We therefore cannot directly employ existing results on hypergraph matching in random uniform hypergraphs to deduce that ${\cal H}$ has a perfect matching.
We can also not use extremal results for this purpose as the degrees in ${\cal H}$ are not large enough
(the density of ${\cal H}$ is a very small constant).
The construction of $P$ is performed in two stages. The first stage is a randomized stage,
which we call the {\em absorber stage}.
It consists of $k-2$ iterative steps. We will show that with positive probability,
the absorber stage ``succeeds''.
In terms of the hypergraph ${\cal H}$, this absorber consists of a subgraph of ${\cal H}$ and
of subgraphs of the $r$-partite $r$-uniform hypergraph projected by ${\cal H}$ to $\cup_{i=1}^r A_i^*$.
It has the property that whenever we want to extend an already found set of $n-k+1\;$ $r$-cliques
to a set of $n-k+1\;$ $(r+1)$-cliques, we can use the absorber to match any remaining vertices of $A^*_{r+1}$ that
become ``dangerous'' and are difficult to match. So, it is an absorber in the sense defined
by R{\"o}dl, Ruci{\'n}ski, and Szemer{\'e}di \cite{RRS-2009} (see also \cite{krivelevich-1997,EGP-1991}
for earlier papers implicitly using this concept).
Given that the absorber construction succeeded,
the second stage is a deterministic, $k-1$ steps process which we call the
{\em gradual matching stage}.
Note that for the case $k=2$, there is only a single step in the second stage part,
and no first stage part, and this amounts to the simple proof for the case $k=2$ given in Section 3.
For larger $k$, the gradual matching stage uses (at its $r$'th step) both the absorber and a Hall-type maximum matching argument in order to extend a perfect $r$-set to a prefect $(r+1)$-set.
In Section 5 we describe the absorber stage. In Section 6 we describe the gradual matching stage.

\section{Absorber stage}

The purpose of this section is to construct sets $Q^*_r$ for $2 \le r < k$ such that
$Q^*_r$ is a set of pairwise-disjoint $r$-cliques in $\cup_{i=1}^r A_i^*$
(and no vertex appears in a $Q^*_j$ and a $Q^*_\ell$ if $\ell \neq j$).
$Q^*_r$ will have some nice properties that guarantee that vertices of $A^*_{r+1}$ that become
``problematic'' during the $r$'th step of the iterative construction of the $n-k+1$ disjoint cliques, can still be
matched to an element of $Q^*_r$ to form an $(r+1)$-clique. Hence, we call $\cup_{r=2}^{k-1} Q^*_r$ the
{\em absorber}.

The construction process proceeds in $k-2$ steps.
We describe Step $r$ for $r=2,\ldots,k-1$ (there is no ``Step 1'') in which we construct $Q^*_r$.
From here onwards, let $m=\lceil \delta n \rceil$.

We first describe Step 2, which is the first step.
Let $Q_2$ denote the set of all edges of $L_\pi(T)$ with one endpoint in $A^*_1$ and the other in
$A^*_2$ and which form a friendly $2$-clique (recall definition \ref{d:friendly-clique}).
Pick at random precisely $m$ {\em pairwise disjoint} elements of $Q_2$,
and denote the set of selected edges by $Q^*_2$.
At a general step $r$, we consider the set $Q_r$ of all friendly $r$-cliques of $L_\pi(T)$
induced by $\cup_{i=1}^r A^*_r$. We remove from $Q_r$ all the cliques containing vertices
that appear in $\cup_{i=2}^{r-1} Q^*_i$ (note: there are
$2m+3m+\cdots+(r-1)m = (\binom{r}{2}-1)m$ such vertices), and denote the resulting set by $Q'_r$.
We pick at random precisely $m$ pairwise disjoint
elements of $Q'_r$, and denote the set of selected $r$-cliques by $Q^*_r$. Our main lemma in this section is the following:
\begin{lemma}\label{l:successful}
	With positive probability the following holds for all $2 \le r < k$.
	For each $v \in A^*_{r+1}$ there are at least $\delta \rho n/4$ elements of $Q^*_r$ such that each of them, together with $v$, induces an $(r+1)$-clique in $L_\pi(T)$.
\end{lemma}
\begin{proof}
	There are at most $(k-2)(n-k+1) < nk$ vertices in $\cup_{i=3}^k A^*_i$ so
	it suffices to prove that for a given vertex $v \in A^*_{r+1}$,
	the probability that it does not have $\delta \rho n/4$ elements of $Q^*_r$ as stated in the lemma is smaller than
	$1/(kn)$, and then the result follows by the union bound.
	
	Let, therefore, $v \in A^*_{r+1}$ and recall that $v$ is a friendly vertex.
	For $i=1,\ldots,r$, let $S_i \subseteq A^*_i$ 
	be the set of neighbors of $v$ in $L_\pi(T)$.
	Since $v$ is friendly, it has at least $n/17$ vertices in $A_i$, so
	$|S_i| \ge n/17 -(k-1) \ge n/18=\mu n$.
	Let  $M$ denote the set of friendly $r$-cliques in $L_\pi(T)$ induced by $\cup_{i=1}^r S_i$.
	By Property 2, $Y(T,\pi,S_1,\ldots,S_r)$ does not occur, hence
	$|M| \ge 0.5\mu^r n^r 2^{-\binom{r}{2}}$.
	As trivially, each vertex of $\cup_{i=1}^r S_i$ appears in at most $n^{r-1}$ cliques of $M$,
	and since the number of vertices appearing in $\cup_{i=2}^{r-1} Q^*_i$ is only
	$(\binom{r}{2}-1)m$, there is a subset $M^* \subseteq M$ of size at least
	\begin{align}
	|M^*| & \ge |M| - (\binom{r}{2}-1)m n^{r-1} \nonumber\\
	& \ge n^r(0.5\mu^r 2^{-\binom{r}{2}} - 2(\binom{r}{2}-1)\delta) \nonumber\\
	& \ge n^r 0.25\mu^r 2^{-\binom{r}{2}} \nonumber\\
	& \ge n^r 0.25\mu^k 2^{-\binom{k}{2}} \nonumber\\
	& =\rho n^r \label{e:M*}
	\end{align}
	where each clique of $M^*$ does not contain vertices of $\cup_{i=2}^{r-1} Q^*_i$, so
	$M^* \subseteq Q'_r$. In the last displayed equation we have used that
	$m=\lceil \delta n \rceil \le 2\delta n$ and the definitions of $\rho$ and $\delta$ in \ref{d:constants}.
	
	Consider the random selection process of $Q^*_r$ from $Q'_r$.
	It consists of $m$ stages
	where at each stage we pick at random an element of $Q'_r$ out of all elements that do not intersect
	elements selected at previous stages. We want to prove that with high probability, a constant fraction of 
	the selected elements are from $M^*$. This, of course, is plausible since $M^*$ amounts to
	a constant fraction of the elements of $Q'_r$. To formalize this, it is convenient to use a  martingale. Let $X_i$ be the indicator variable which equals $1$ if at the $i$'th stage, an element
	of $M^*$ has been picked. Then $Y_m = X_1+\cdots+X_m$ is the number of elements of $M^*$ that have been picked. Let $Y_0=E[Y_m]$ and for $i=1,\ldots,m$, let $Y_i = E[Y_m | X_1,\ldots,X_{i}]$.
	Then $Y_0,\ldots,Y_m$ is a Doob martingale by definition, and observe also that since $X_i$ is
	an indicator variable, $|Y_{i+1}-Y_i| \le 1$ for $i=0,\ldots,m-1$. Hence, by Lemma \ref{l:azuma}
	for $\lambda > 0$,
	$$
	\Pr[Y_m < Y_0-\lambda \sqrt{m}] < e^{-\lambda^2/2}\;.
	$$ 
	We next estimate $Y_0 = E[Y_m]$. Clearly, $E[X_1]=|M^*|/|Q'_r| \ge \rho n^r/n^r=\rho$ where we have used (\ref{e:M*}).
	To lower bound $E[X_i]$, we lower bound it conditioned on $X_1+\cdots + X_{i-1}$.
	As $X_1+\cdots + X_{i-1} \le i-1$, there are at most $(i-1)r n^{r-1}$ elements of $M^*$ which
	contain vertices of previously selected elements, and therefore
	$E[X_i | X_1+\cdots X_{i-1}] \ge (\rho n^r - (i-1)rn^{r-1})/n^r$.
	Since $i-1 \le m-1 \le \delta n$ we have that
	$$
	E[X_i | X_1+\cdots+ X_{i-1}] \ge \frac{\rho n^r - (i-1)rn^{r-1}}{n^r} \ge \rho-\delta r \ge \frac{\rho}{2}
	$$
	where we have used that $\delta r \le \delta k \le \rho/k \le \rho/2$.
	As the lower bound $\rho/2$ does not depend on $X_1+\cdots+ X_{i-1}$ we have that
	$E[X_i] \ge \rho/2$. Therefore,
	$$
	Y_0 = E[Y_m] = \sum_{i=1}^m E[X_i] \ge \frac{m\rho}{2} \ge \frac{\delta \rho n}{2}\;.
	$$
	Hence, for $\lambda = \sqrt{m}\rho/5$ we have that
	\begin{align*}
	\Pr[Y_m < \delta \rho n/4] & = \Pr[Y_m < \delta \rho n/2 - \delta \rho n/4]\\
	& \le \Pr[Y_m \le Y_0 - \delta \rho n/4]\\
	& \le \Pr[Y_m \le Y_0 - m \rho/5] \\
	& = \Pr[Y_m \le Y_0 - \lambda \sqrt{m}]\\
	& < e^{-\lambda^2/2}\\
	& \le e^{-\Theta(n)}\\
	& < \frac{1}{nk}\;.
	\end{align*}
\end{proof}

So, from here onwards we fix the absorber $\cup_{r=2}^{k-1}Q^*_r$.
It has the property that for all $2 \le r < k$, and for
each vertex $v \in A^*_{r+1}$ there are at least $\delta \rho n/4$ elements of $Q^*_r$ 
such that each of them, together with $v$, induces an $(r+1)$-clique in $L_\pi(T)$.
Furthermore, each element of $Q^*_r$ is a friendly $r$-clique and
no vertex appears more than once in the absorber. Finally, $|Q_r^*|=m$ for all $2 \le r \le k-1$.

\section{Gradual matching stage}

This stage proceeds in $k-1$ steps. Starting with $r=1$,
in step $r$ we construct a perfect $(r+1)$-set $P_{r+1}$ which induces
$n-k+1$ disjoint $(r+1)$-cliques and furthermore, $P_{r+1}$,
when restricted to $\cup_{i=1}^r A_i$, is $P_r$.

Before describing the steps, we need to specify certain subsets, which depend on the absorber.
Let $B_1 \subset A^*_1$ be the set of vertices that do not appear in $\cup_{i=2}^{k-1} Q^*_i$
(namely the vertices of $A_1^*$ that do not appear in any element of the absorber).
For $r = 2,\ldots,k$, let $B_r \subseteq A^*_r$ be the set of vertices that do not appear in $\cup_{i=r}^{k-1} Q^*_i$. Observe that $B_k=A^*_k$, and since $|Q^*_i|=m$, we have that 
$|B_1|=n-k+1-(k-2)m$ and for $2 \le r \le k$ we have that $|B_r|=n-k+1-(k-r)m$. In particular, $|B_1|=|B_2|$
and
\begin{equation}\label{e:br}
|B_r| \ge n-k+1-(k-2)m \ge n-k(m-1) \ge n(1-\delta k)
\end{equation}
for $1 \le r \le k$ as we recall that $m=\lceil \delta n \rceil$.
Also, as every $v \in B_r \subseteq A^*_r$
is a friendly vertex, it has at least $n/17$ neighbors (in $L_\pi(T)$) in $A_j$ for $j \neq r$,
so it also has at least $n/17 - n\delta k \ge n/18=\mu n$ neighbors in $B_j$.

We describe the first step as it is simpler since it does not depend on parameters of previous steps.
Consider the induced bipartite graph $H_1$ of $L_\pi(T)$ with one side being $B_1$ and the the other side being $B_2$. By the previous paragraph, the minimum degree of $H_1$ is at least $\mu n$.
This, together with Property 1 (applied to $r=1$) and Hall's Theorem suffices to guarantee a perfect matching in $H_1$. But this is not enough. We need to make sure that the matching edges
that we choose will behave ``nicely'' with respect to future steps.
For this, we need to first discard potentially bad edges of $H_1$.
\begin{definition}[Friendly $H_1$ edge]
	An edge $uv$ of $H_1$ is {\em friendly} if for all $3 \le t \le k$, the number of common neighbors of $u$ and $v$ in $B_t$ is at least $n/d^2$. Otherwise, $uv$ is {\em unfriendly}.
\end{definition}
\begin{lemma}\label{l:friendly-edge}
	Every vertex of $B_1 \cup B_2$ is incident with fewer than $dk$ unfriendly edges of $H_1$.
\end{lemma}
\begin{proof}
	Suppose some $v \in B_1$ is incident with $dk$ unfriendly edges (the argument if $v \in B_2$ is identical).
	Let these edges be $vu_1,\ldots,vu_{dk}$ where $u_i \in B_2$.
	As all these edges are unfriendly, there is some $3 \le t \le k$ such that at least $d$ of them
	are unfriendly with respect to $B_t$, namely they have fewer than $n/d^2$ common neighbors in $B_t$.
	Suppose these are $vu_1,\ldots,vu_d$.
	Since $v$ has at least $n/18=\mu n$ neighbors in $B_t$, at least
	$\mu n-d(n/d^2)=\mu n - n/d$ of the vertices of $B_t$ are non-neighbors of all of $u_1,\ldots,u_d$.
	Without loss of generality, assume that $\pi(u_1) < \pi(u_2) < \cdots < \pi(u_d)$.
	So there are at least $(\mu n-n/d)/(d+1)$ vertices of $B_t$, all appearing in $\pi$
	after $u_i$ and before $u_{i+1}$ (or else all before $u_1$ or else all after $u_d$)
	and none of them are neighbors of $u_1,\ldots,u_d$. But according to Property $4$
	applied with $s=t$, $\ell=2$ and $W(j) = (-)$ if $j \le i$ and $W(j)=(+)$ if $d \ge j \ge i+1$, the number
	of such vertices is at most $n2^{-d}+n^{2/3}$.
	But by the definition of $d$ in \ref{d:constants},
	$$
	\frac{\mu n-n/d}{d+1} > \frac{n}{2^d}+n^{2/3}\;,
	$$
	a contradiction.
\end{proof}
Let $H_1^*$ be the spanning subgraph of $H_1$ obtained after removing all unfriendly edges of $H_1$.
Then, by Lemma \ref{l:friendly-edge}, the minimum degree of $H_1^*$ is at least
$\mu n-dk \ge n/19$.
We can now easily prove using Hall's Theorem that $H_1^*$ has a perfect matching.
Indeed, for $R \subseteq B_1$, we must show that $|N_{B_2}(R)| \ge |R|$ where $N_{B_2}(R)$ is the
set of neighbors (in $H^*_1$) of $R$ in $B_2$.
This trivially holds if $|R| \le n/19$ or $|R| \ge |B_1|-n/19$ by the minimum degree of
$H_1^*$. For $|R|$ within these two values, let $S=B_2 \setminus N_{B_2}(R)$.
Then, there is no edge between $R$ and $S$ in $H_1^*$ and since by Lemma \ref{l:friendly-edge}
there are fewer than $dkn$ unfriendly edges of $H_1$, we have that
there are fewer than
$dkn=\Theta(n)$ edges between $R$ and $S$ in $H_1$, thus also in $L_\pi(T)$.
But since $|R| \ge n/19 \ge \epsilon n$, we must have by Property 1 that $|S| \le \epsilon n$.
But then,
$$
|N_{B_2}(R)| = |B_2|-|S| = |B_1|-|S| \ge |B_1| - \epsilon n \ge |B_1|-n/19 \ge |R|\;.
$$

We construct the perfect $2$-set $P_2$ as follows.
We take a perfect matching in $H_1^*$.
We then take from each element of the absorber $\cup_{i=2}^{k-1} Q^*_i$ the pair of vertices with one endpoint in
$A_1^*$ and the other in $A_2^*$. Observe that since the elements of $Q^*_i$ are $i$-cliques, then each
chosen pair is a matching edge. 
Finally, we arbitrarily pair the remaining $k-1$ vertices of
$A_1 \setminus A_1^*$ with the remaining $k-1$ vertices of $A_2 \setminus A_2^*$. These $k-1$ pairs
are not necessarily edges of $L_\pi(T)$. Altogether $P_2$ is a perfect $2$-set,
containing a subset $P_2^*$ of $n-k+1$ elements that are edges of $L_\pi(T)$ matching the vertices of $A_1^*$ with the vertices of $A_2^*$.

In fact, the $P_2$ that we have just constructed satisfies the case $r=2$ of the following definition.
\begin{definition}[Extendable perfect $r$-set]\label{d:extend-r-set}
	Let $2 \le r \le k$. A perfect $r$-set $P_r$ is called {\em extendable} if the following holds.
	\begin{enumerate}
		\item 
		There is a subset $P_r^* \subset P_r$ of order $n-k+1$ such that each element of $P_r^*$ induces an $r$-clique in $L_\pi(T)$.
		Furthermore, each element of $P_r^*$ contain a single vertex from $A_i^*$ for $i=1,\ldots,r$.
		\item
		The first $r$-vertices of each element of $\cup_{i=r}^{k-1} Q^*_i$, form an element of $P_r^*$.
		In particular $Q_r^* \subset P_r^*$.
		\item
		For each element $p=(a_1,\ldots,a_r)$ of $P_r^*$, and for each $r < t \le k$, the number of common neighbors of $a_1,\ldots,a_r$ in $A_t$ is at least $n/d^{2r-2}$.
	\end{enumerate}
\end{definition}
Observe that in the above definition, if $r=k$ then an extendable perfect $k$-set just needs to satisfy the first requirement, as the other two become empty requirements. In particular, an extendable perfect $k$-set
contains $n-k+1$ elements, each of which is a $k$-clique in $L_\pi(T)$.
So, if we can find an extendable perfect $k$-set, we have proved Theorem \ref{t:1-equiv}.
Indeed, the following lemma shows that we can.
\begin{lemma}\label{l:extend}
	Let $2 \le r \le k-1$. If there is an extendable perfect $r$-set, then there is an extendable perfect
	$(r+1)$-set. Consequently, if there is an extendable perfect $2$-set, then Theorem \ref{t:1-equiv} holds.
\end{lemma}
Before proving Lemma \ref{l:extend}, we first need to verify that $P_2$ that we have constructed above,
is an extendable perfect $2$-set.
\begin{lemma}\label{l:p2}
	$P_2$ is an extendable perfect $2$-set.
\end{lemma}
\begin{proof}
	The first two requirements follow immediately from our construction. For the third requirement,
	consider some element of $p \in P_{2}^*$ where $p=(a_1,a_2)$. Then there are two cases.
	Case 1: $p$ is the prefix of some element $p' \in Q^*_i$ where $2 \le i \le k-1$ (if $i=2$ then $p=p'$
	is just an element of $Q^*_{2}$). But since $p'$ is an element of the absorber,
	it is a friendly $i$-clique. But this means that for $2 < t \le k$, the vertices
	$a_1,a_2$ have (in $L_\pi(T)$) at least $n/8 > n/d^2$ common neighbors in $A_t$,
	so the third requirement is met. Case 2: $p$ is the result of the matching in $H^*_1$ which matched
	$a_1 \in B_1$ with $a_2 \in B_2$. But then, $a_1a_2$ is a friendly
	$H_1$-edge. But this means that if $2 < t \le k$, the number of common neighbors of $a_1,a_2$
	in $B_t$ is at least $n/d^2$ and again the third requirement is met.
\end{proof}

\noindent
{\em Proof of Lemma \ref{l:extend}.}
Suppose that we are given an extendable perfect $r$-set $P_r$. We wish to use it and construct an
extendable perfect $(r+1)$-set $P_{r+1}$.

Consider the graph $L_\pi(P_r,T)$ and recall that, in particular, it is a bipartite graph with one part being $P_r$ and the other part being $A_{r+1}$.
Let $H_r$ be the induced bipartite subgraph of
$L_\pi(P_r,T)$ where one side is $B_{r+1}$ and the other side is the set of elements of $P^*_r$ that do
not contain vertices of $\cup_{i=r+1}^{k-1} Q^*_i$. Denote this other side by $J_r$
and observe that all $m$ elements of $Q_r^*$ remain elements of $J_r$.
First observe that both sides are of the same size: $|B_{r+1}|=|J_r|=(n-k+1)-m(k-1-r)$.

Once again we would like to prove that $H_r$ has a perfect matching but also that the matching edges
the we choose are ``nice''. For this, we first establish a minimum degree bound for $H_r$.
\begin{lemma}\label{l:mindeg-hr}
	The minimum degree of $H_r$ is at least $\delta \rho n/4$.
\end{lemma}
\begin{proof}
	Consider first some vertex $v \in B_{r+1}$ (recall that the sides of $H_r$ are $B_{r+1}$ and $J_r$).
	Then, since all elements of $Q_r^*$ are elements of $J_r$, we have by the property of the
	absorber that $v$ has at last $\delta \rho n/4$ neighbors in $H_r$ (in fact, already in $Q_r^*$).
	Consider next a vertex $p \in J_r$. Then $p=(a_1,\ldots,a_r)$ and $\{a_1,\ldots,a_r\}$ induce an $r$-clique.
	Since $J_r \subset P^*_r$, we have by the third property of extendable perfect $r$-sets,
	that the number of common neighbors of $a_1,\ldots,a_r$ in $A_{r+1}$ is at least $n/d^{2r-2}$.
	But by (\ref{e:br}) we have that $|B_{r+1}| \ge n(1-\delta k)$ so 
	the number of common neighbors of $a_1,\ldots,a_r$ in $B_{r+1}$ is at least $n/d^{2r-2}-\delta k n$.
	In other words, $p$ has at least $n/d^{2r-2}-\delta k n$ neighbors in $H_r$.
	Since $1/d^{2r-2}-\delta k \ge 1/d^{2k} - \delta k \ge 2\delta k - \delta k \ge \delta k \ge \delta \rho/4$,
	the lemma follows.
\end{proof}

\begin{definition}[Friendly $H_r$ edge]
	Let $2 \le r \le k-1$. An edge $pv$ of $H_r$ where $p=(a_1,\ldots,a_r)$ and $v \in B_{r+1}$ is {\em friendly} if for each $r+2 \le t \le k$, the number of common neighbors of $\{a_1,\ldots,a_r,v\}$ in $B_t$ is at least $n/d^{2r}$. Otherwise, $pv$ is {\em unfriendly}. Observe that every edge of $H_{k-1}$ is friendly.
\end{definition}
\begin{lemma}\label{l:friendly-edge-hr}
	Let $2 \le r \le k-1$. Every vertex of $H_r$ is incident with fewer than $dk$ unfriendly edges of $H_r$.
\end{lemma}
\begin{proof}
	As the lemma is trivial for $r=k-1$ (all edges of $H_{k-1}$ are friendly), we assume $2 \le r \le k-2$.
	The vertices of $H_r$ are $J_r \cup B_{r+1}$. Assume first that $v \in B_{r+1}$
	is incident with $dk$ unfriendly edges of $H_r$.
	Let these edges be $p_1v,\ldots,p_{dk}v$ where $p_i \in J_r$.
	As they are all unfriendly, there is some $r+2 \le t \le k$ such that at least $d$ of them
	are unfriendly with respect to $B_t$.
	Suppose these are $p_1v,\ldots,p_dv$.
	Let $p_i=(a_{i,1},\ldots,a_{i,r})$ with $a_{i,j} \in A_j$.
	Since $v$ has at least $n/18=\mu n$ neighbors in $B_t$, at least
	$\mu n-dn/d^{2r}=\mu n - n/d^{2r-1}$ of them are not in the common neighborhood of all of $p_1,\ldots,p_d$.
	Let this set be $B^* \subseteq B_t$. So, $|B^*| \ge \mu n-n/d^{2r-1}$ and each $u \in B^*$ has the property
	that for all $i=1,\ldots,d$, there is some $a_{i,j}$ such that $a_{i,j}$ and $u$ are not neighbors
	in $L_\pi(T)$.
	
	Let $Z$ denote the set of $dr$ vertices of $p_1,\ldots,p_d$. So, there is a $B^{**} \subseteq B^*$
	with $|B^{**}| \ge |B^*|/(dr+1)$ and a partition $Z_1 \cup Z_2 = Z$, such that every $u \in B^{**}$
	appears in $\pi$ after all vertices of $Z_1$ and before all vertices of $Z_2$ (possibly $Z_1=\emptyset$
	or $Z_2=\emptyset$). Let $W_i \in \{+,-\}^r$ be the vector with $W_i(j)=(+)$ if $a_{i,j} \in Z_2$
	and $W_i(j)=(-)$ if $a_{i,j} \in Z_1$. Let $\hat{W}=(W_1,\ldots,W_d)$ and let $\hat{p}=(p_1,\ldots,p_d)$.
	Than, every vertex $u \in B^{**}$ is $\hat{W}$-inconsistent with $\hat{p}$.
	Hence, $B^{**} \subseteq I_{\hat{W}}(\hat{p},t)$. By Property 3, 
	$|B^{**} | \le n(1-1/2^r)^d+n^{2/3}$.
	On the other hand, $|B^{**}| \ge |B^*|/(dr+1) \ge (\mu n-n/d^{2r-1})/(dr+1)$.
	But by Definition \ref{d:constants},
	$$
	\frac{\mu n-n/d^{2r-1}}{dr+1}  > n(1-1/2^r)^d+n^{2/3}\;,
	$$
	a contradiction.
	
	Assume next that $p=(a_1,\ldots,a_r) \in J_r$ is incident with $dk$ unfriendly edges of $H_r$.
	Let these edges be $pv_1,\ldots,pv_{dk}$ where $v_i \in B_{r+1}$.
	As they are all unfriendly, there is some $r+2 \le t \le k$ such that at least $d$ of them
	are unfriendly with respect to $B_t$.
	Suppose these are $pv_1,\ldots,pv_d$.
	Since $p \in J_r \subseteq P^*_r$, we have by the property of an extendable $r$-set that
	there are at least $n/d^{2r-2}$ common neighbors of $\{a_1,\ldots,a_r\}$ in $B_t$.
	So, there is a set $B^* \subseteq B_t$ with $|B^*| \ge n/d^{2r-2}-d(n/d^{2r})=n/d^{2r-2}-n/d^{2r-1}$
	such that each $u \in B^*$ is a non-neighbor of each of $v_1,\ldots,v_d$.
	Assume that $\pi(v_1) < \pi(v_2) \cdots < \pi(v_d)$. So, there is a subset
	$|B^{**}| \ge |B^*|/(d+1)$ such that every $u \in B^{**}$
	appears in $\pi$ after $v_i$ and before $v_{i+1}$ (or else before $v_1$ or else after $v_d$).
	But according to property 4 applied with $s=t$, $\ell=r+1$ and $W(j)=(-)$ if $j \le i$ and
	$W(j)=(+)$ if  $d \ge j \ge i+1$, then number of such vertices is at most $n2^{-d}+n^{2/3}$.
	On he other hand, $|B^{**}| \ge |B^*|/(d+1) \ge (n/d^{2r-2}-n/d^{2r-1})/(d+1)$.
	But by Definition \ref{d:constants}
	$$
	\frac{n/d^{2r-2}-n/d^{2r-1}}{d+1} > \frac{n}{2^d}+n^{2/3}\;,
	$$
	a contradiction.
\end{proof}
Let therefore $H_r^*$ be the spanning subgraph of $H_r$ obtained after removing all unfriendly $H_r$-edges.
Then, by Lemmas \ref{l:mindeg-hr} and \ref{l:friendly-edge-hr}, the minimum degree of $H_r^*$ is at least
$\delta \rho n/4 - dk \ge \delta \rho n/5$.
We can now easily prove using Hall's Theorem that $H_r^*$ has a perfect matching.
Indeed, for $R \subseteq J_r$, we must show that $|N_{B_{r+1}}(R)| \ge |R|$ where $N_{B_{r+1}}(R)$ is the
set of neighbors (in $H^*_r$) of $R$ in $B_{r+1}$.
This trivially holds if $|R| \le \delta \rho n/5$ or $|R| \ge |B_{r+1}|-\delta \rho n/5$ by the minimum degree of
$H_r^*$. For $|R|$ within these two values, let $S=B_{r+1} \setminus N_{B_{r+1}}(R)$.
Then, there is no edge between $R$ and $S$ in $H_r^*$ and hence by Lemma \ref{l:friendly-edge-hr}
there are fewer than $dkn=\Theta(n)$ edges between $R$ and $S$ in $H_r$, thus also in $L_\pi(P_r,T)$.
But since $|R| \ge \delta \rho n/5  \ge \epsilon n$, we must have by Property 1 that
$|S| \le \epsilon n$. But then,
$$
|N_{B_{r+1}}(R)| = |B_{r+1}|-|S| \ge |B_{r+1}| - \epsilon n \ge |B_{r+1}|-\delta \rho n/5 \ge |R|\;.
$$

We construct $P_{r+1}$ as follows.
We take a perfect matching in $H_r^*$. Each such matching edge is of the form $pv$ with $p \in J_r \subset P^*_r$
and $v \in B_{r+1}$, so $p \cup \{v\}$ induces an $(r+1)$-clique.
We then take from each element of $\cup_{i=r+1}^{k-1} Q^*_i$ the $(r+1)$-tuple of vertices with one endpoint in $A_i^*$ for $i=1,\ldots,r+1$.
Observe that since the elements of $Q^*_i$ are $i$-cliques, then each
chosen $(r+1)$-tuple is an $(r+1)$-clique.
Finally, we arbitrarily match the remaining $k-1$ vertices of
$A_{r+1} \setminus A^*_{r+1}$ with the $k-1$ elements of $P_r \setminus P^*_r$
into $k-1$ additional $(r+1)$-tuples.
These $k-1$ additional $(r+1)$-tuples
are not necessarily $(r+1)$-cliques of $L_\pi(T)$.
Altogether $P_{r+1}$ is a perfect $(r+1)$-set
containing a subset $P_{r+1}^*$ of $n-k+1$ elements that are $(r+1)$-cliques of $L_\pi(T)$
matching the vertices of $A_i^*$ for $i=1,\ldots,r+1$.
In particular, the first two requirements in the definition of extendable perfect $(r+1)$-sets
are satisfied.

To complete the proof of Lemma \ref{l:extend} it remains to show that the third requirement 
in the definition of extendable perfect $(r+1)$-sets is also met.
Consider some element of $p \in P_{r+1}^*$ where $p=(a_1,\ldots,a_{r+1})$. Then there are two cases.
Case 1: $p$ is the prefix of some element $p' \in Q^*_i$ where $r+1 \le i \le k-1$ (if $i=r+1$ then $p=p'$
is just an element of $Q^*_{r+1}$). But since $p'$ is an element of the absorber,
it is a friendly $i$-clique. But this means that for $r+1 < t \le k$, the vertices
$a_1,\ldots,a_{r+1}$ have (in $L_\pi(T)$) at least $n/2^{r+2} > n/d^{2r}$ common neighbors in $A_t$,
so the third requirement is met. Case 2: $p$ is the result of the matching in $H^*_r$ which matched
$p'=(a_1,\ldots,a_r) \in J_r$ with $a_{r+1} \in B_{r+1}$. But then, $p'a_{r+1}$ is a friendly
$H_r$-edge. But this means that if $r+1 < t \le k$, the number of common neighbors of $\{a_1,\ldots,a_r,a_{r+1}\}$ in $B_t$ is at least $n/d^{2r}$ and again the third requirement is met.
\qed

\section*{Acknowledgment}
We thank the reviewers for very helpful suggestions.

\bibliographystyle{plain}

\bibliography{references}

\end{document}